\documentclass[amstex,12pt, amssymb]{article}

\usepackage{mathtext}
\usepackage[cp1251]{inputenc}
\usepackage[T2A]{fontenc}
\usepackage[dvips]{graphicx}
\usepackage{amsmath}
\usepackage{amssymb}
\usepackage{amsxtra}
\usepackage{latexsym}
\usepackage{ifthen}

\newcounter{lemma}[section]

\newcounter{corollary}[section]

\newcounter{remark}[section]

\newcounter{theorem}[section]

\newcounter{proposition}[section]

\newcounter{example}

\numberwithin{equation}{section}

\begin{document}

\markboth{V. DESYATKA, E.~SEVOST'YANOV}{\centerline{ON
BOUNDARY-NON-PRESERVING MAPPINGS ...}}

\def\cc{\setcounter{equation}{0}
\setcounter{figure}{0}\setcounter{table}{0}}

\overfullrule=0pt


\author{VICTORIA DESYATKA, OLEKSANDR DOVHOPIATYI, \\ EVGENY SEVOST'YANOV}

\title{
{\bf ON BOUNDARY-NON-PRESERVING MAPPINGS WITH INTEGRAL CONSTRAINTS}}

\date{\today}
\maketitle

\begin{abstract}
This manuscript is devoted to the study of mappings, satisfying the
upper weighted Poletsky inequality. We study the case where the
boundary of the domain may not be preserved under the mapping and,
besides that, the majorant from the above inequality satisfies
constraints of the integral-type. Under certain additional
conditions on the definition domain and the corresponding cluster
sets, we prove that families of above mappings are equicontinuous in
the closure of this domain.
\end{abstract}


\section{Introduction}

The paper is devoted to the study of mappings with bounded and
finite distortion, see, e.g., \cite{Cr$_2$}, \cite{MRV$_1$},
\cite{MRSY$_1$}--\cite{MRSY$_2$}, \cite{PSS}, \cite{RV},
\cite{SalSt}, \cite{Vu} and~\cite{Va}. We are interested in mappings
that have upper estimates for the distortion of the modulus of
families of paths as part of their definition, see, e.g.,
\cite{Cr$_1$}--\cite{Cr$_3$}, \cite{MRSY$_2$} and \cite{Sev$_3$}. As
can be seen, the study of the local behavior of such mappings
predominantly assumes one of the following properties: the mappings
are either homeomorphisms or open, discrete, and closed. The
papers~\cite{Cr$_1$}, \cite{NP}, \cite{Va},
\cite{RS$_1$}--\cite{RS$_3$} can be classified in the first group,
while publications \cite{Cr$_3$}, \cite{Sr}, \cite{Vu} belong to the
second. We should note that the openness, discreteness, and
closeness of mappings imply the preservation of the boundary under
them, and vice versa, open, discrete, and boundary-preserving
mappings are closed; see, e.g., \cite[Theorem]{Vu}. Non-closed
(boundary-non-preserving) mappings remain poorly investigated. This
is primarily due to the rather complex methodology of their study,
since the behavior of the so-called path liftings under them
correlates poorly with the behavior of the corresponding mappings.
We should also note that there are quite a lot of non-closed
mappings. For example, an analytic function $f(z)=z^4,$ on the
complex plane is a closed mapping in ${\Bbb D}=\{z\in {\Bbb C}:
|z|<1\}$ but in the disk $B(1, 1)=\{z\in {\Bbb C}: |z-1|<1\}$ it is
no longer such. Due to the well-known theorems of
Casorati-Weierstrass or Picard, even an isolated singularity of an
analytic function will be its essential singularity only in the case
when the mapping is not closed (closed mappings, in particular,
preserve the boundary of the domain, and, at the same time, the
cluster set of an analytic function with an essential singularity is
the entire extended complex plane).

\medskip
Relatively recently, we began studying non-closed mappings,
proposing the following methodology: reducing the study of
boundary-non-preserving mappings to their investigation in smaller
domains in which the mappings are already closed. This, in turn, can
be achieved by requiring certain additional restrictions on the
cluster set; see~\cite{DS$_1$}. In this manuscript, we continue
research in this direction. As in~\cite{DS$_1$}, the main object of
study is the boundary behavior of mappings, as well as the global
behavior of their families in the closure of the domain. A classical
result on this topic states the following (see~\cite{NP}).

\medskip
{\bf Theorem~{\bf A.} (N\"{a}kki--Palka).}{\it\, Let $\frak{F}$ be a
family of $K$-quasiconformal mappings of a domain $D\ne
\overline{{\Bbb R}^n}$ onto a domain $D^{\,\prime}$ and let either
$D$ or $D^{\,\prime}$ be quasiconformally collared on the boundary.
Then $\frak{F}$ is uniformly equicontinuous if and only if each
$f\in\frak{F}$ can be extended to a continuous mapping of
$\overline{D}$ onto $\overline{D^{\,\prime}}$ and
$\inf\limits_{\frak{F}} h(f(A))>0$ for some continuum $A$ in $D.$}

\medskip Here $h(f(A))$ denotes the chordal diameter of the set
$f(A)\subset\overline{{\Bbb R}^n},$ see
e.g.~\cite[Definition~12.1]{Va}.

\medskip
We have obtained some analogs of this statement, including for
homeomorphisms and open discrete and closed mappings, see, e.g.
\cite{Sev$_4$} and \cite{SevSkv$_1$}. In the most general case, when
the mappings are neither homeomorphisms nor closed mappings, we have
obtained the following statement, see below. Let us first formulate
the definitions and designations that will appear below.

\medskip A Borel function $\rho:{\Bbb R}^n\,\rightarrow [0,\infty] $
is called {\it admissible} for the family $\Gamma$ of paths $\gamma$
in ${\Bbb R}^n,$ if the relation
\begin{equation*}\label{eq1.4}
\int\limits_{\gamma}\rho (x)\, |dx|\geqslant 1
\end{equation*}
holds for all (locally rectifiable) paths $ \gamma \in \Gamma.$ In
this case, we write: $\rho \in {\rm adm} \,\Gamma .$ Let $p\geqslant
1,$ then {\it $p$-modulus} of $\Gamma $ is defined by the equality
\begin{equation*}\label{eq1.3gl0}
M_p(\Gamma)=\inf\limits_{\rho \in \,{\rm adm}\,\Gamma}
\int\limits_{{\Bbb R}^n} \rho^p (x)\,dm(x)\,.
\end{equation*}
We set $M(\Gamma):=M_n(\Gamma).$ Let $x_0\in {\Bbb R}^n,$
$0<r_1<r_2<\infty,$
\begin{equation}\label{eq1ED}
S(x_0,r) = \{ x\,\in\,{\Bbb R}^n : |x-x_0|=r\}\,, \quad B(x_0, r)=\{
x\,\in\,{\Bbb R}^n : |x-x_0|<r\}\end{equation}
and
\begin{equation}\label{eq1**}
A=A(x_0, r_1,r_2)=\left\{x\,\in\,{\Bbb R}^n:
r_1<|x-x_0|<r_2\right\}\,.\end{equation}
Given sets $E,$ $F\subset\overline{{\Bbb R}^n}$ and a domain
$D\subset {\Bbb R}^n$ we denote by $\Gamma(E,F,D)$ a family of all
paths $\gamma:[a,b]\rightarrow \overline{{\Bbb R}^n}$ such that
$\gamma(a)\in E,\gamma(b)\in\,F $ and $\gamma(t)\in D$ for $t \in
(a, b).$ Let $S_i=S(x_0, r_i),$ $i=1,2,$ where spheres $S(x_0, r_i)$
centered at $x_0$ of the radius $r_i$ are defined in~(\ref{eq1ED}).
Let $Q:{\Bbb R}^n\rightarrow {\Bbb R}$ be a Lebesgue measurable
function satisfying the condition $Q(x)\equiv 0$ for $x\in{\Bbb
R}^n\setminus D.$ Let $p\geqslant 1.$ A mapping $f:D\rightarrow
\overline{{\Bbb R}^n}$ is called a {\it ring $Q$-mapping at the
point $x_0\in \overline{D}\setminus \{\infty\}$ with respect to
$p$-modulus}, if the condition
\begin{equation} \label{eq2*!A}
M_p(f(\Gamma(S_1, S_2, D)))\leqslant \int\limits_{A\cap D} Q(x)\cdot
\eta^p (|x-x_0|)\, dm(x)
\end{equation}
holds for some $r_0(x_0)>0,$ all $0<r_1<r_2<r_0$ and all Lebesgue
measurable functions $\eta:(r_1, r_2)\rightarrow [0, \infty]$ such
that
\begin{equation}\label{eqA2}
\int\limits_{r_1}^{r_2}\eta(r)\,dr\geqslant 1\,.
\end{equation}
The relation~(\ref{eq2*!A}) may be defined at $x_0=\infty$ with the
help of the inversion: by the definition, we say that $f$
satisfies~(\ref{eq2*!A}) whenever $\widetilde{f}=f\circ \psi$
satisfies~(\ref{eq2*!A}) at $x_0=0,$ while
$\psi(x)=\frac{x}{|x|^2}.$

\medskip A domain $D\subset {\Bbb R}^n,$ $n\geqslant 2,$ is called a
{\it uniform} domain with respect to $p$-modulus, $p\geqslant 1,$
if, for each $r>0,$ there is $\delta>0$ such that $M_{p}(\Gamma(F,
F^{\,*}, D))\geqslant\delta$ whenever $F$ and $F^{\,*}$ are continua
of $D$ with $h(F)\geqslant r$ and $h(F^{\,*})\geqslant r.$ Let $I$
be some set of indices. Domains $D_i,$ $i\in I,$ are said to be {\it
equi-uniform} domains with respect to $p$-modulus, if, for $r>0,$
the modulus condition above is satisfied by each $D_i$ with the same
number $\delta.$ It should be noted that the proposed concept of a
uniform domain has, generally speaking, no relation to definition,
introduced for the uniform domain in Martio-Sarvas sense~\cite{MSa}.

\medskip
Following to~\cite{Ri}, a {\it condenser} in ${\Bbb R}^n$ is a pair
$E=(A, C),$ where $A$ is open in ${\Bbb R}^n$ and $C\ne \varnothing$
is a compact subset of $A.$ If $1\leqslant p <\infty,$ the {\it
$p$-capacity of $E$} is defined by
$${\rm cap}_p E =\inf\limits_{u} \int\limits_A |\nabla u|^p\,dm\,,$$
where the infimum is taken over all nonnegative functions $u$ in
${\rm ACL}^p(A)$ with compact support in $A$ and $u|_C \geqslant 1.$
Let $F$ be a compact set in ${\Bbb R}^n.$ We say that $F$ is of {\it
$p$-capacity zero} if ${\rm cap}_p\,(A, F)=0$ for some (and hence
for all) bounded open set $A\supset F.$ An arbitrary set $E\subset
{\Bbb R}^n$ is of {\it $p$-capacity zero} if the same is true for
every compact subset of $E.$ In this case we write ${\rm cap}_p\,E =
0$ (${\rm cap}\, E=0$ if $p=n$), otherwise ${\rm cap}_p E > 0.$

\medskip
Let $Q:{\Bbb R}^n\rightarrow [0,\infty]$ be a Lebesgue measurable
function. We set
$$Q^{\,\prime}(x)=\left\{
\begin{array}{rr}
Q(x), &   Q(x)\geqslant 1\,, \\
1,  &  Q(x)<1\,.
\end{array}
\right.$$ Denote by $q^{\,\prime}_{x_0}$ the mean value of
$Q^{\,\prime}(x)$ over the sphere $|x-x_0|=r$, that means,
\begin{equation}\label{eq32*B}
q^{\,\prime}_{x_0}(r):=\frac{1}{\omega_{n-1}r^{n-1}}
\int\limits_{|x-x_0|=r}Q^{\,\prime}(x)\,d{\mathcal H}^{n-1}\,.
\end{equation}

\medskip
Given $p\geqslant 1,$ $\delta>0,$ closed sets $E, E_*, F$ in
$\overline{{\Bbb R}^n},$ $n\geqslant 2,$ a domain $D\subset {\Bbb
R}^n$ and a Lebesgue measurable function $Q:D\rightarrow [0,
\infty]$ let us denote by $\frak{R}^{E_*, E, F}_{Q, \delta, p}(D)$ a
family of open discrete mappings $f:D\rightarrow \overline{{\Bbb
R}^n}\setminus F$ satisfying the
conditions~(\ref{eq2*!A})--(\ref{eqA2}) at any point $x_0\in
\overline{D}$ such that:

\medskip
1) $C(f, \partial D)\subset E_*,$

\medskip
2) for each component $K$ of $D^{\,\prime}_f\setminus  E_*,$
$D^{\,\prime}_f:=f(D),$ there is a continuum $K_f\subset K$ such
that $h(K_f)\geqslant \delta$ and $h(f^{\,-1}(K_f), \partial
D)\geqslant \delta>0,$

\medskip
3) $f^{\,-1}(E_*)\subset E.$

\medskip
{\bf Theorem~B.} {\it\, Let $p\geqslant 1,$ let $D$ be a domain in
${\Bbb R}^n,$ $n\geqslant 2.$ Assume that:

\medskip
1) the set $E$ is nowhere dense in $D,$ and $D$ is finitely
connected on $E,$ i.e., for any $z_0\in E$ and any neighborhood
$\widetilde{U}$ of $z_0$ there is a neighborhood
$\widetilde{V}\subset \widetilde{U}$ of $z_0$ such that $(D\cap
\widetilde{V})\setminus E$ consists of finite number of components;

\medskip
2) for any $x_0\in\partial D$ there is $m=m(x_0)\in {\Bbb N},$
$1\leqslant m<\infty$ such that the following is true: for any
neighborhood $U$ of $x_0$ there is a neighborhood $V\subset U$ of
$x_0$ and  such that:

\medskip
2a) $V\cap D$ is connected,

\medskip
2b) $(V\cap D)\setminus E$ consists at most of $m$ components.

\medskip
Let for $p=n$ the set $F$ have positive capacity, and for $n-1<p<n$
it is an arbitrary closed set.

\medskip
Suppose that, for any $x_0\in\partial D$ at least one of the
following conditions is satisfied: $3_1)$ a function $Q$ has a
finite mean oscillation at $x_0;$ $3_2)$
$q_{x_0}(r)\,=\,O\left(\left[\log{\frac1r}\right]^{n-1}\right)$ as
$r\rightarrow 0;$ $3_3)$ the condition
\begin{equation*}\label{eq6D}
\int\limits_{0}^{\delta(x_0)}\frac{dt}{t^{\frac{n-1}{p-1}}
q_{x_0}^{\,\prime\,\frac{1}{p-1}}(t)}=\infty
\end{equation*}
holds for some $\delta(x_0)>0,$ where $q_{x_0}^{\,\prime}(t)$ is
defined in~(\ref{eq32*B}).

Let the family of all components of $D^{\,\prime}_f\setminus E_*$ is
equi-uniform over $f\in\frak{R}^{E_*, E, F}_{Q, \delta, p}(D)$ with
respect to $p$-modulus. Then every $f\in\frak{R}^{E_*, E, F}_{Q,
\delta, p}(D)$ has a continuous extension to $\partial D$ and the
family $\frak{R}^{E_*, E, F}_{Q, \delta, p}(\overline{D}),$
consisting of all extended mappings $\overline{f}:
\overline{D}\rightarrow \overline{{\Bbb R}^n},$ is equicontinuous
in~$\overline{D}.$ }

\medskip
Our goal is to establish Theorem~B for slightly different conditions
on the function~$Q,$ in particular, to study the case where the
function $Q$ may vary. Similar conditions for theorems of other
kinds have been considered by us previously, in particular for
theorems on the compactness of classes of solutions to the Dirichlet
problem; see, e.g., \cite{RS$_3$}, \cite{RSY} and \cite{DS$_2$}.

\medskip
Given $p\geqslant 1,$ $\delta>0,$ closed sets $E, E_*, F$ in
$\overline{{\Bbb R}^n},$ $n\geqslant 2,$ a domain $D\subset {\Bbb
R}^n,$ a number $M>0$ and an increasing convex function
$\Phi:\overline{{\Bbb R^{+}}}\rightarrow \overline{{\Bbb R^{+}}}$ we
denote by $\frak{R}^{E_*, E, F}_{M, \Phi, \delta, p}(D)$ a family of
open discrete mappings $f:D\rightarrow \overline{{\Bbb
R}^n}\setminus F$ for which there exists a Lebesgue measurable
function $Q=Q_f:D\rightarrow [0, \infty]$ satisfying the
conditions~(\ref{eq2*!A})--(\ref{eqA2}) at the point $x_0\in
\overline{D}$ for some $r_0=r_0(x_0)>0$ and any $0<r_1<r_2<r_0$ for which
\begin{equation}\label{eq1D}
\int\limits_{\Omega}\Phi(Q_f(x))\cdot\frac{dm(x)}{(1+|x|^2)^n}\leqslant
M
\end{equation}
and such that

\medskip
1) $C(f, \partial D)\subset E_*,$

\medskip
2) for each component $K$ of $D^{\,\prime}_f\setminus  E_*,$
$D^{\,\prime}_f:=f(D),$ there is a continuum $K_f\subset K$ such
that $h(K_f)\geqslant \delta$ and $h(f^{\,-1}(K_f), \partial
D)\geqslant \delta>0,$

\medskip
3) $f^{\,-1}(E_*)\subset E.$

\medskip
The following statement is true, cf. {\bf Theorems~A} and {\bf B.}

\medskip
\begin{theorem}\label{th1}
{\,\it Let $p\geqslant 1,$ let $D$ be a domain in
${\Bbb R}^n,$ $n\geqslant 2.$ Assume that:

\medskip
1) the set $E$ is nowhere dense in $D,$ and $D$ is finitely
connected on $E,$ i.e., for any $z_0\in E$ and any neighborhood
$\widetilde{U}$ of $z_0$ there is a neighborhood
$\widetilde{V}\subset \widetilde{U}$ of $z_0$ such that $(D\cap
\widetilde{V})\setminus E$ consists of finite number of components;

\medskip
2) for any $x_0\in\partial D$ there is $m=m(x_0)\in {\Bbb N},$
$1\leqslant m<\infty$ such that the following is true: for any
neighborhood $U$ of $x_0$ there is a neighborhood $V\subset U$ of
$x_0$ and  such that:

\medskip
2a) $V\cap D$ is connected,

\medskip
2b) $(V\cap D)\setminus E$ consists at most of $m$ components.

\medskip
Let for $p=n$ the set $F$ have positive capacity, and for $n-1<p<n$
it is an arbitrary closed set.

\medskip
Suppose that, the relation
\begin{equation}\label{eq2} \int\limits_{\delta_0}^{\infty}
\frac{d\tau}{\tau\left[\Phi^{-1}(\tau)\right]^{\frac{1}{p-1}}}=
\infty
\end{equation}
holds for some $\delta_0>\tau_0:=\Phi(0).$ Let the family of all components
of $D^{\,\prime}_f\setminus E_*$ is
equi-uniform over $f\in\frak{R}^{E_*, E, F}_{M, \Phi, \delta, p}(D))$ with
respect to $p$-modulus. Then every $f\in\frak{R}^{E_*, E, F}_{M, \Phi, \delta, p}(D)$
has a continuous extension to $\partial D$ and the
family $\frak{R}^{E_*, E, F}_{M, \Phi, \delta, p}(\overline{D}),$
consisting of all extended mappings $\overline{f}:
\overline{D}\rightarrow \overline{{\Bbb R}^n},$ is equicontinuous
in~$\overline{D}.$}
\end{theorem}

\section{Preliminaries}

The following lemma was proved in \cite[Lemma~2.1]{DS$_1$}.

\begin{lemma}\label{lem2}
{\it\, Let $D$ be a domain in ${\Bbb R}^n,$ $n\geqslant 2,$ and let
$x_0\in \partial D.$ Assume that $E$ is closed and nowhere dense in
$D,$ and $D$ is finitely connected on $E,$ i.e., for any $z_0\in E$
and any neighborhood $\widetilde{U}$ of $z_0$ there is a
neighborhood $\widetilde{V}\subset \widetilde{U}$ of $z_0$ such that
$(D\cap \widetilde{V})\setminus E$ consists of finite number of
components.

In addition, assume that the following condition holds: for any
neighborhood $U$ of $x_0$ there is a neighborhood $V\subset U$ of
$x_0$ such that:

\medskip
a) $V\cap D$ is connected,

\medskip
b) $(V\cap D)\setminus E$ consists at most of $m$ components,
$1\leqslant m<\infty.$

Let $x_k, y_k\in D\setminus E,$ $k=1,2,\ldots ,$ be a sequences
converging to $x_0$ as $k\rightarrow\infty.$ Then there are
subsequences $x_{k_l}$ and $y_{k_l},$ $l=1,2,\ldots ,$ belonging to
some sequence of neighborhoods $V_l,$ $l=1,2,\ldots ,$ of the point
$x_0$ such that $V_l\subset B(x_0, 2^{\,-l}),$ $l=1,2,\ldots ,$ and,
in addition, any pair $x_{k_l}$ and $y_{k_l}$ may be joined by a
path $\gamma_l$ in $V_l\cap D,$ where $\gamma_l$ contains at most
$m-1$ points in~$E.$ }
\end{lemma}

Let $D\subset {\Bbb R}^n,$ $f:D\rightarrow {\Bbb R}^n$ be a discrete
open mapping, $\beta: [a,\,b)\rightarrow {\Bbb R}^n$ be a path, and
$x\in\,f^{\,-1}(\beta(a)).$ A path $\alpha: [a,\,c)\rightarrow D$ is
called a {\it maximal $f$-lifting} of $\beta$ starting at $x,$ if
$(1)\quad \alpha(a)=x\,;$ $(2)\quad f\circ\alpha=\beta|_{[a,\,c)};$
$(3)$\quad for $c<c^{\prime}\leqslant b,$ there is no a path
$\alpha^{\prime}: [a,\,c^{\prime})\rightarrow D$ such that
$\alpha=\alpha^{\prime}|_{[a,\,c)}$ and $f\circ
\alpha^{\,\prime}=\beta|_{[a,\,c^{\prime})}.$ Here and in the
following we say that a path $\beta:[a, b)\rightarrow
\overline{{\Bbb R}^n}$ converges to the set $C\subset
\overline{{\Bbb R}^n}$ as $t\rightarrow b,$ if $h(\beta(t),
C)=\sup\limits_{x\in C}h(\beta(t), C)\rightarrow 0$ at $t\rightarrow
b.$ The following is true (see~\cite[Lemma~3.12]{MRV$_2$}).

\medskip
\begin{proposition}\label{pr3}
{\it\, Let $f:D\rightarrow {\Bbb R}^n,$ $n\geqslant 2,$ be an open
discrete mapping, let $x_0\in D,$ and let $\beta: [a,\,b)\rightarrow
{\Bbb R}^n$ be a path such that $\beta(a)=f(x_0)$ and such that
either $\lim\limits_{t\rightarrow b}\beta(t)$ exists, or
$\beta(t)\rightarrow \partial f(D)$ as $t\rightarrow b.$ Then
$\beta$ has a maximal $f$-lifting $\alpha: [a,\,c)\rightarrow D$
starting at $x_0.$ If $\alpha(t)\rightarrow x_1\in D$ as
$t\rightarrow c,$ then $c=b$ and $f(x_1)=\lim\limits_{t\rightarrow
b}\beta(t).$ Otherwise $\alpha(t)\rightarrow \partial D$ as
$t\rightarrow c.$}
\end{proposition}

\medskip
The following statement was proved in~\cite[Lemma~2.1]{Sev$_5$}.

\medskip
\begin{lemma}\label{lem1}
{\it\, Let $1\leqslant p\leqslant n,$ and let $\Phi:[0,
\infty]\rightarrow [0, \infty] $ be a strictly increasing convex
function such that the relation~(\ref{eq2}) holds for some
$\delta_0>\tau_0:=\Phi(0).$ Let $\frak{Q}$ be a family of functions
$Q:{\Bbb R}^n\rightarrow [0, \infty]$ such that
\begin{equation}\label{eq5A}
\int\limits_D\Phi(Q(x))\frac{dm(x)}{\left(1+|x|^2\right)^n}\
\leqslant M_0<\infty
\end{equation}
for some $0<M_0<\infty.$ Now, for any $0<r_0<1$ and for every
$\sigma>0$ there exists $0<r_*=r_*(\sigma, r_0, \Phi)<r_0$ such that
\begin{equation}\label{eq4}
\int\limits_{\varepsilon}^{r_0}\frac{dt}{t^{\frac{n-1}{p-1}}q^{\frac{1}{p-1}}_{x_0}(t)}\geqslant
\sigma\,,\qquad \varepsilon\in (0, r_*)\,,
\end{equation}
for any $Q\in \frak{Q}.$ }
\end{lemma}

\medskip
\begin{remark}\label{rem1}
We set
$$Q^{\,\prime}(x)=\begin{cases}Q(x), & Q(x)\geqslant 1\\ 1, & Q(x)<1\end{cases}\,.$$
Observe that $Q^{\,\prime}(x)$ satisfies (\ref{eq5A}) up to a
constant, see, e.g., proof of Theorem~1.3 in \cite{Sev$_5$}.
\end{remark}

\section{Proof of Theorem~\ref{th1}}

Mainly we use our approach applied in~\cite{DS$_1$}. The
equicontinuity of $\frak{R}^{E_*, E, F}_{M, \Phi, \delta, p}(D)$ in
$D$ follows by~\cite[Theorems~1.1, 1.2]{Sev$_5$}.

\medskip
Now, we prove the following statement: for any $x_0\in \partial D,$
$x_0\ne \infty,$ and for every $\varepsilon>0$ there is
$\delta=\delta(\varepsilon, x_0)>0$ such that
\begin{equation}\label{eq1}
h(f(x), f(y))<\varepsilon
\end{equation}
for any $f\in\frak{R}^{E_*, E, F}_{M, \Phi, \delta, p}(D)$ and any
$x, y\in B(x_0, \delta)\cap D.$

\medskip
Suppose the opposite. Then there is $x_0\in\partial D,$
$\varepsilon_0>0,$ a sequences $x_k, x^{\,\prime}_k\in D,$ $x_k,
x^{\,\prime}_k\rightarrow x_0$ as $k\rightarrow \infty,$ and $f_k\in
\frak{R}^{E_*, E, F}_{M, \Phi, \delta, p}(D)$ such that
\begin{equation}\label{eq13A}
h(f_k(x_k), f_k(x^{\,\prime}_k))\geqslant \varepsilon_0/2\,.
\end{equation}
By the assumption~2), there exists a sequence of neighborhoods
$V_k\subset B(x_0, 2^{\,-k}),$ $k=1,2,\ldots ,$ such that $V_k\cap
D$ is connected and $(V_k\cap D)\setminus E$ consists of $m$
components, $1\leqslant m<\infty.$

We note that the points $x_k$ and $x^{\,\prime}_k,$ $k=1,2,\ldots, $
may be chosen such that $x_k, x^{\,\prime}_k\not\in E.$ Indeed,
since under condition~1) the set $E$ is nowhere dense in $D,$ there
exists a sequence $x_{ki}\in D\setminus E,$ $i=1,2,\ldots ,$ such
that $x_{ki}\rightarrow x_k$ as $i\rightarrow\infty.$ Put
$\varepsilon>0.$ Due to the continuity of the mapping $f_k$ at the
points $x_k$ and $x^{\,\prime}_k,$ for the number $k\in {\Bbb N}$
there are numbers $i_k, j_k\in {\Bbb N}$ such that $h(f_k(x_{ki_k}),
f_k(x_k))<\frac{1}{2^k}$ and $h(f_k(x^{\,\prime}_{kj_k}),
f_k(x^{\,\prime}_k))<\frac{1}{2^k}.$ Now, by~(\ref{eq13A}) and by
the triangle inequality,
\begin{gather}\nonumber
h(f_k(x_{ki_k}), f_k(x^{\,\prime}_{kj_k}))\geqslant
h(f_k(x^{\,\prime}_k), f_k(x_{ki_k}))-h(f_k(x^{\,\prime}_k),
f_k(x^{\,\prime}_{kj_k})) \\
\label{eq14B}\geqslant h(f_k(x_k), f_k(x^{\,\prime}_k))-h(f_k(x_k),
f_k(x_{ki_k}))-h(f_k(x^{\,\prime}_k), f_k(x^{\,\prime}_{kj_k}))
\\
\nonumber\geqslant\varepsilon_0/2-\frac{2}{2^k}\geqslant
\varepsilon_0/4
\end{gather}
for sufficiently large $k.$ Due to~(\ref{eq14B}), we may assume that
$x_k, x^{\,\prime}_k\not\in E,$  as required.

\medskip
Now, by Lemma~\ref{lem1} there are subsequences $x_{k_l}$ and
$x^{\,\prime}_{k_l},$ $l=1,2,\ldots ,$ belonging to some sequence of
neighborhoods $V_l,$ $l=1,2,\ldots ,$ of the point $x_0$ such that
${\rm diam\,}V_l\rightarrow 0$ as $l\rightarrow\infty$ and, in
addition, any pair $x_{k_l}$ and $x^{\,\prime}_{k_l}$ may be joined
by a path $\gamma_l$ in $V_l\cap D,$ where $\gamma_l$ contains at
most $m-1$ points in $E.$ Without loss of generality, we may assume
that the same sequences $x_k$ and $y_k$ satisfy properties mentioned
above. Let $\gamma_k:[0, 1]\rightarrow D,$ $\gamma_k(0)=x_k$ and
$\gamma_k(1)=y_k,$ $k=1,2,\ldots .$

\medskip
Observe that, the path $f_k(\gamma_k)$ contains not more than $m-1$
points in $E_*.$ In the contrary case, there are at least $m$ such
points $b_{1}=f_k(\gamma_k(t_1)), b_{2}=f_k(\gamma_k(t_2)),\ldots,
b_m=f_k(\gamma_k(t_m)),$ $0\leqslant t_1\leqslant t_2\leqslant
\ldots\leqslant t_m\leqslant 1.$ But now the points
$a_{1}=\gamma_k(t_1), a_{2}=\gamma_k(t_2),\ldots, a_m=\gamma_k(t_m)$
are in $f_k^{\,-1}(E_*)\subset E$ and simultaneously belong to
$\gamma_k.$ This contradicts the definition of $\gamma_k.$

\medskip
Let
$$b_{1}=f_k(\gamma_k(t_1)), b_{2}=f_k(\gamma_k(t_2))\quad,\ldots,\quad
b_l=f_k(\gamma_k(t_l))\,,$$ $$t_0:=0\leqslant t_1\leqslant
t_2\leqslant \ldots\leqslant t_l\leqslant 1=t_{l+1},\qquad
0\leqslant l\leqslant m-1\,,$$
be points in $f_k(\gamma_k)\cap E_*.$ By the relation~(\ref{eq13A})
and due to the triangle inequality,
\begin{equation}\label{eq6C}
\varepsilon_0/2\leqslant h(f_k(x_k),
f_k(y_k))\leqslant\sum\limits_{r=0}^{l} h(f_k(\gamma_k(t_r)),
f_k(\gamma_k(t_{r+1}))\,.
\end{equation}
It follows from~(\ref{eq6C}) that, there is $1\leqslant
r=r(k)\leqslant m-1$ such that such that
\begin{equation}\label{eq7F}
h(f(\gamma_k(t_{r(k)})), f_k(\gamma_k(t_{r(k)+1})))\geqslant
\varepsilon_0/(2(l+1))\geqslant \varepsilon_0/(2m)\,.
\end{equation}
Observe that, the set $G_k:=|\gamma_k|_{(t_{r(k)}, t_{r(k)+1})}|$
belongs to $D^{\,\prime}\setminus E.$

\medskip
Observe that, $G_k$ contains some a continuum $\widetilde{G}_k$ with
$h(\widetilde{G}_k)\geqslant \varepsilon_0/4m$ for any $k\in {\Bbb
N}.$ Indeed, by the construction of $G_k,$ there is a sequence of
points $a_{s, k}:=f_k(\gamma_k(p_s))\rightarrow
w_k:=f(\gamma_k(t_{r(k)}))$ as $s\rightarrow \infty$ and $b_{s,
k}=f_k(\gamma_k(q_s))\rightarrow f(\gamma_k(t_{r(k)+1}))$ as
$s\rightarrow \infty,$ where $t_{r(k)}< p_s<q_s<t_{r(k)+1}$ and
$a_{s, k}, b_{s, k}\in G_k.$
By the triangle inequality and by~(\ref{eq7F})
$$
h(a_{s, k}, b_{s, k})\geqslant h(b_{s, k}, w_k)-h(w_k, a_{s,
k})\geqslant
$$
\begin{equation*}\label{eq11A}
h(f_k(\gamma_k(t_{r(k)+1})), w_k)-h(f_k(\gamma_k(t_{r(k)+1})), b_{s,
k})-h(w_k, a_{s, k})\geqslant
\end{equation*}
$$\varepsilon_0/(2m)-h(f_k(\gamma_k(t_{r(k)+1}), b_{s,
k}))-h(w_k, a_{s, k})\,.$$
Since $a_{s, k}:=f_k(\gamma_k(p_s))\rightarrow
w_k:=f_k(\gamma_k(t_{r(k)}))$ as $s\rightarrow \infty$ and $b_{s,
k}=f_k(\gamma_k(q_s))\rightarrow f_k(\gamma_k(t_{r(k)+1}))$ as
$s\rightarrow \infty,$ it follows from the last inequality that
there is $s=s(k)\in {\Bbb N}$ such that
\begin{equation*}\label{eq12B}
h(a_{s(k), k}, b_{s(k), k})\geqslant\varepsilon_0/(4m)\,.
\end{equation*}
Now, we set
$$\widetilde{G}_k:=f_k(\gamma_k)|_{[p_{s(k)}, q_{s(k)}]}\,.$$
In other words, $\widetilde{G}_k$ is a part of the path
$f_k(\gamma_k)$ between points $a_{s(k), k}$ and $b_{s(k), k}.$
Since $\widetilde{G}_k$ is a continuum in
$D^{\,\prime}_{f_k}\setminus E_*,$ there is a component $K_k$ of
$D^{\,\prime}_{f_k}\setminus E_*,$ containing $K_k.$ Let us apply
the definition of equi-uniformity for the sets $\widetilde{G}_k$ and
$K_{f_k}$ in $K_k$ (here $K_{f_k}$ is a continuum from the
definition of the class $\frak{R}^{E_*, E, F}_{M, \Phi, \delta,
p}(D)$, in particular, $h(K_{f_k})\geqslant \delta$). Due to this
definition, for the number $\delta_*:=\min\{\delta,
\varepsilon/4m\}>0$ there is $P>0$ such that
\begin{equation}\label{eq1F}
M_p(\Gamma(\widetilde{G}_k, K_{f_k}, K_k))\geqslant P
>0\,, \qquad k=1,2,\ldots\,.
\end{equation}
Let us to show that, the relation~(\ref{eq1F}) contradicts with the
definition of $f$ in~(\ref{eq2*!A}). Indeed, let us denote by
$\Gamma_k$ the family of all half-open paths $\beta_k:[a,
b)\rightarrow \overline{{\Bbb R}^n}$ such that $\beta_k(a)\in
\widetilde{G}_k,$ $\beta_k(t)\in K_k$ for all $t\in [a, b)$ and,
moreover, $\lim\limits_{t\rightarrow b-0}\beta_k(t):=B_k\in
K_{f_k}.$ Obviously, by~(\ref{eq1F})
\begin{equation}\label{eq4C}
M_p(\Gamma_k)=M_p(\Gamma(\widetilde{G}_k, K_{f_k}, K_k))\geqslant P
>0\,, \qquad k=1,2,\ldots\,.
\end{equation}
Consider the family $\Gamma_k^{\,\prime}$ of all maximal
$f_k$-liftings $\alpha_k:[a, c)\rightarrow D$ of the family
$\Gamma_k$ starting at $|\gamma_k|;$ such a family exists by
Proposition~\ref{pr3}.

Observe that, the situation when $\alpha_k\rightarrow \partial D$ as
$k\rightarrow\infty$ is impossible. Suppose the opposite: let
$\alpha_k(t)\rightarrow \partial D$ as $t\rightarrow c.$ Let us
choose an arbitrary sequence $\varphi_m\in [0, c)$ such that
$\varphi_m\rightarrow c-0$ as $m\rightarrow\infty.$ Since the space
$\overline{{\Bbb R}^n}$ is compact, the boundary $\partial D$ is
also compact as a closed subset of the compact space. Then there
exists $w_m\in
\partial D$ such that
\begin{equation}\label{eq7G}
h(\alpha_k(\varphi_m), \partial D)=h(\alpha_k(\varphi_m), w_m)
\rightarrow 0\,,\qquad m\rightarrow \infty\,.
\end{equation}
Due to the compactness of $\partial D$, we may assume that
$w_m\rightarrow w_0\in \partial D$ as $m\rightarrow\infty.$
Therefore, by the relation~(\ref{eq7G}) and by the triangle
inequality
\begin{equation}\label{eq8C}
h(\alpha_k(\varphi_m), w_0)\leqslant h(\alpha_k(\varphi_m),
w_m)+h(w_m, w_0)\rightarrow 0\,,\qquad m\rightarrow \infty\,.
\end{equation}
On the other hand,
\begin{equation}\label{eq9C}
f_k(\alpha_k(\varphi_m))=\beta_k(\varphi_m)\rightarrow \beta(c)
\,,\quad m\rightarrow\infty\,,
\end{equation}
because by the construction the path $\beta_k(t),$ $t\in [a, b],$
lies in $K_k\subset D^{\,\prime}_{f_k}\setminus E_*$ together with
its finite ones points. At the same time, by~(\ref{eq8C})
and~(\ref{eq9C}) we have that $\beta_k(c)\in C(f_k,
\partial D)\subset E_*$ by the definition of the class
$\frak{R}^{E_*, E, F}_{M, \Phi, \delta, p}(D).$ The inclusions
$\beta_k\subset D^{\,\prime}\setminus E_*$ and $\beta_k(c)\in E_*$
contradict each other.

\medskip
Therefore, by Proposition~\ref{pr3} $\alpha_k\rightarrow x_1\in D$
as $t\rightarrow c-0,$ and $c_b$ and $f_k(\alpha_k(b))=f_k(x_1).$ In
other words, the $f_k$-lifting $\alpha_k$ is complete, i.e.,
$\alpha_k:[a, b]\rightarrow D.$ Besides that, it follows from that
$\alpha_k(b)\in f_k^{\,-1}(K_{f_k}).$

\medskip
Again, by the definition of the class $\frak{R}^{E_*, E, F}_{M,
\Phi, \delta, p}(D),$
\begin{equation}\label{eq13B}
h(f_k^{\,-1}(K_{f_k}), \partial D)\geqslant \delta>0\,.
\end{equation}
Since $x_0\ne\infty,$ it follows from~(\ref{eq13B}) that
\begin{equation}\label{eq14C}
f_k^{\,-1}(K_{f_k})\subset D\setminus B(x_0, r_0)
\end{equation}
for any $k\in {\Bbb N}$ and some $r_0>0.$ Let $k$ be such that
$2^{\,-k}<\varepsilon_0.$ We may consider that $\varepsilon_0<r_0.$
Due to~(\ref{eq14C}), we may show that
\begin{equation}\label{eq15A}
\Gamma_k^{\,\prime}>\Gamma(S(x_0, 2^{\,-k}), S(x_0, \varepsilon_0),
D)
\end{equation}
(see \cite[Theorem~1.I.5, \S46]{Ku}).
Now, we set
\begin{equation}\label{eq9} I_k=I(x_0, 2^{\,-k},
\varepsilon_0)=\int\limits_{2^{\,-k}}^{\varepsilon_0}\
\frac{dr}{r^{\frac{n-1}{p-1}}\widetilde{q}_{k,
x_0}^{\frac{1}{p-1}}(r)}\,,
\end{equation}
where
$$\widetilde{q}_{k, x_0}(r)=\frac{1}{\omega_{n-1}r^{n-1}} \int\limits_{S(x_0,
r)}\widetilde{Q}_{f_{m_k}}(x)\,d\mathcal{H}^{n-1}\,$$ and
$$\widetilde{Q}_{f_{m_k}}(x)=\begin{cases}
Q_{f_{m_k}}(x), & Q(x)\geqslant 1\\ 1, & Q(x)<1\end{cases}\,.$$

By Lemma~\ref{lem1} and Remark~\ref{rem1}
\begin{equation}\label{eq2_1}
I_k=\int\limits_{2^{\,-k}}^{\varepsilon_0}\
\frac{dr}{r^{\frac{n-1}{p-1}}\widetilde{q}_{k,
x_0}^{\frac{1}{p-1}}(r)}\rightarrow\infty
\end{equation}
as $k\rightarrow\infty.$
Since $\widetilde{q}_{k, x_0}(r)\geqslant 1$ for a.e. $r$ and
by~(\ref{eq2_1}), $0<I_k<\infty$ for sufficiently large $k\in {\Bbb
N}.$ Set
$$\psi_k(t)= \left \{\begin{array}{rr}
1/[t^{\frac{n-1}{p-1}}\widetilde{q}_{k, x_0}^{\frac{1}{p-1}}(t)], &
t\in (2^{\,-k}, \varepsilon_0)\ ,
\\ 0,  &  t\notin (2^{\,-k}, \varepsilon_0)\ .
\end{array} \right. $$
Let $\eta_k(t)=\psi_k(t)/I_k$ for $t\in (2^{\,-k}, \varepsilon_0)$
and $\eta_k(t)=0$ otherwise. Now, $\eta_k$ satisfies~(\ref{eqA2})
for $r_1=2^{\,-k}$ and $r_2=\varepsilon_0.$ Now, by~(\ref{eq2*!A}),
(\ref{eq15A}), (\ref{eq2_1}) and by Fubini's theorem,
\begin{equation}\label{eq3_1}
M_p(f_k(\Gamma^{\,\prime}_k)\leqslant\frac{1}{I^{p}_k}\int\limits_{A(x_0,
2^{\,-k}, \varepsilon_0)}
\widetilde{Q}_{f_{m_k}}(x)\cdot\psi_k^p(|x-x_0|)\,dm(x)=\frac{\omega_{n-1}}{I^{p-1}_k}\rightarrow
0
\end{equation}
as $k\rightarrow\infty.$ Thus, by~(\ref{eq3_1}) and
(\ref{eq15A})--(\ref{eq9}) we obtain that
\begin{equation}\label{eq3A}
M_p(\Gamma_k)= M_p(f_k(\Gamma_k^{\,\prime}))\leqslant
\Delta(k)\rightarrow 0
\end{equation}
as $k\rightarrow \infty.$ However, the relation~(\ref{eq3A})
together with the inequality~(\ref{eq4C}) contradict each other,
which proves~(\ref{eq1}).

\medskip
Observe that, the possibility of a continuous extension of any $f\in
\frak{R}^{E_*, E, F}_{M, \Phi, \delta, p}(D)$ at $x_0,$ $x_0\ne
\infty,$ follows from~(\ref{eq1}), as well. Indeed, otherwise we
have $h(f(x_k), f(y_k))\geqslant \varepsilon_0$ for some sequences
$x_k, y_k\rightarrow x_0$ as $k\rightarrow\infty,$ but this
contradicts with~(\ref{eq1}).

\medskip
Now, we prove that $\frak{R}^{E_*, E, F}_{M, \Phi, \delta,
p}(\overline{D})$ is equicontinuous at $x_0$ with a respect to
$\overline{D}.$ Suppose the opposite. Then there is $x_0\in\partial
D,$ $\varepsilon_0>0,$ a sequence $x_k\rightarrow x_0,$
$x_k\in\overline{D},$ and $f_k\in \frak{R}^{E_*, E, F}_{M, \Phi,
\delta, p}(\overline{D})$ such that
\begin{equation}\label{eq12C}
h(f_k(x_k), f_k(x_0))\geqslant \varepsilon_0\,.
\end{equation}
Since $f_k$ has a continuous extension to $\partial D,$ we may
consider that $x_k\in D.$ In addition, it follows from~(\ref{eq12C})
that we may find a sequence $x^{\,\prime}_k\in D,$ $k=1,2,\ldots ,$
such that $x^{\,\prime}_k\rightarrow x_0$ and
\begin{equation}\label{eq13C}
h(f_k(x_k), f_k(x^{\,\prime}_k))\geqslant \varepsilon_0/2\,.
\end{equation}
The relation~(\ref{eq13C}) contradicts with~(\ref{eq1}). The
contradiction obtained above completes the proof for $x_0\ne\infty.$

\medskip
Finally, if $x_0=\infty,$ we set $\widetilde{f}_k:=\psi\circ f_k,$
$k=1,2,\ldots ,$ where $\psi(y)=\frac{y}{|y|^2}.$ By the assumption
of Theorem~\ref{th1}, the mappings $\widetilde{f}_k$ satisfy the
relations~(\ref{eq2*!A})--(\ref{eqA2}) at the origin with a new
function $\widetilde{Q}(x):=Q_{f_k}\left(\frac{x}{|x|}\right).$
The conformal change of the variables $x=\frac{y}{|y|^2}$
corresponds to the jacobian $J(x, y)=\frac{1}{|y|^{2n}}.$ So,
observe that,
\begin{equation}\label{eq3A_1} \int\limits_{{\Bbb
R}^n}\Phi\left(Q_{f_k}\left(\frac{x}{|x|^2}\right)\right)\cdot\frac{dm(y)}{(1+|x|^2)^n}=
\int\limits_{{\Bbb
R}^n}\Phi(Q_{f_k}(z))\cdot\frac{dm(z)}{(1+|z|^2)^n}<\infty\,.
\end{equation}
The relation~(\ref{eq3A_1}) implies the possibility of applying
above technique to $Q_{f_k}\left(\frac{x}{|x|^2}\right)$ instead of
$Q_{f_k}(x).$ Repeating further all the arguments given from
relation~(\ref{eq13A}) and further for the family of mappings
$\widetilde{f}_k$ instead of $f_k$, we arrive at the contradiction
with~(\ref{eq13A}), which completes the proof of Theorem~\ref{th1}.
$\Box$

\section{Sobolev and Orlicz-Sobolev classes}

Let $\omega$ be an open set in $\overline{{\Bbb R}^k}:={\Bbb
R}^k\cup\{\infty\}$, $k=1,\ldots\,,n-1$. A (continuous) mapping
$S:\omega\rightarrow{\Bbb R}^n$ is called a $k$-dimensional surface
$S$ in ${\Bbb R}^n$. Sometimes we call the image
$S(\omega)\subseteq{\Bbb R}^n$ the surface $S$, too. The number of
preimages
\begin{equation} \label{eq8.2.3}N(S,y)\ =\
\mathrm{card}\,\,S^{-1}(y)=\mathrm{card}\:\{x\in\omega:\
S(x)=y\}\end{equation}
is said to be a {\it multiplicity function} of the surface $S$ at a
point $y\in{\Bbb R}^n$. In other words, $N(S,y)$ denotes the
multiplicity of covering of the point $y$ by the surface $S$. It is
known that the multiplicity function is lower semi continuous, i.e.,
$$N(S,y) \geqslant \liminf_{m\rightarrow\infty}: N(S,y_m)$$
for every sequence $y_m\in{\Bbb R}^n$, $m=1,2,\ldots\,$, such that
$y_m\rightarrow y\in{\Bbb R}^n$ as $m\rightarrow\infty$, see, e.g.,
\cite{RR}, p. 160. Thus, the function $N(S,y)$ is Borel measurable
and hence measurable with respect to every Hausdorff measure ${\cal
H}^k$ (see, e.g., \cite{Sa}, p.~52).

\medskip
If $\rho:{\Bbb R}^n\rightarrow[0,\infty]$ is a Borel function, then
its {\it integral over a $k$-dimensional surface} $S$ in ${\Bbb
R}^n$, $n\geqslant 2,$ is defined by the equality
\begin{equation}\label{eq8.2.5} \int\limits_S \rho\ d{\cal {A}}\
:=\ \int\limits_{{\Bbb R}^n}\rho(y)\:N(S,y)\ d{\cal H}^ky\
.\end{equation} Given a family ${\cal S}$ of such $k$-dimensional
surfaces $S$ in ${\Bbb R}^n$, a Borel function $\ \rho:{\Bbb
R}^n\rightarrow[0,\infty]$ is called {\it admissible} for $\cal S$,
abbr. $\rho\in\mathrm{adm}\,\cal S$, if
\begin{equation}\label{eq8.2.6A}
\int\limits_S\rho^k\ d{\cal{A}}\ \geqslant\ 1\end{equation} for
every $S\in\cal S$. Given $p\in[k,\infty)$, the {\it $p$-modulus} of
$\cal S$ is the quantity
\begin{equation}\label{M} M_p({\cal S})\ =\
\inf_{\rho\in\mathrm{adm}\,\cal S}\int\limits_{{\Bbb R}^n}\rho^p(x)\
dm(x)\,.\end{equation}
The next class of mappings is a generalization of quasiconformal
mappings in the sense of Gehring's ring definition (see \cite{Ge};
cf. \cite[Chapter~9]{MRSY$_2$}). Let $D$ and $D^{\,\prime}$ be
domains in ${\Bbb R}^n,$ $n\geqslant 2$. Suppose that
$x_0\in\overline {D}\setminus\{\infty\}$ and $Q\colon
D\rightarrow(0,\infty)$ is a Lebesgue measurable function. A mapping
$f:D\rightarrow D^{\,\prime}$ is called a {\it lower $Q$-mapping at
a point $x_0$ relative to $p$-modulus} if
\begin{equation}\label{eq1A}
M_p(f(\Sigma_{\varepsilon}))\geqslant \inf\limits_{\rho\in{\rm
ext}_p\,{\rm adm}\Sigma_{\varepsilon}}\int\limits_{D\cap A(x_0,
\varepsilon, r_0)}\frac{\rho^p(x)}{Q(x)}\,dm(x)
\end{equation}
for every spherical ring $A(x_0, \varepsilon, r_0)=\{x\in {\Bbb
R}^n\,:\, \varepsilon<|x-x_0|<r_0\}$, $r_0\in(0,d_0)$,
$d_0=\sup\limits_{x\in D}|x-x_0|$, where $\Sigma_{\varepsilon}$ is
the family of all intersections of the spheres $S(x_0, r)$ with the
domain $D$, $r\in (\varepsilon, r_0)$. If $p=n$, we say that $f$ is
a lower $Q$-mapping at $x_0$. We say that $f$ is a lower $Q$-mapping
relative to $p$-modulus in $A\subset {\Bbb R}^n$ if (\ref{eq1A}) is
true for all $x_0\in A$.

Following \cite[Section~II.10]{Ri}, a {\it condenser} is a pair
$E=(A, C)$ where $A\subset {\Bbb R}^n$ is open and $C$ is non--empty
compact set contained in $A.$ Given $p\geqslant 1$ and a condenser
$E=\left(A,\,C\right),$ we set
%
$${\rm cap}_p\,E\quad=\quad{\rm
cap}_p\,\left(A,\,C\right)\quad=\quad\inf \limits_{u\in
W_0\left(E\right) }\quad\int\limits_{A}\,|\nabla u|^p dm(x)$$
%
%
where $W_0(E)=W_0(A,\,C)$ is a family of all non-negative functions
$u:A\rightarrow {\Bbb R}^1$ such that (1)\quad $u\in
C_0(A),$\quad(2)\quad $u(x)\ge 1$ for $x\in C,$ and (3)\quad $u$ is
$ACL.$ In the above formula
$|\nabla u|={\left(\sum\limits_{i=1}^n\,{\left(\partial_i
u\right)}^2 \right)}^{1/2},$ and ${\rm cap}_p\,E$ is called {\it
$p$-ca\-pa\-ci\-ty} of the condenser $E,$ see in \cite[Section
~II.10]{Ri}.

\medskip
The following lemma was proved in \cite[Lemma~4.4]{Sev$_3$},
cf.~\cite[Lemma~4.2]{SevSkv$_2$}.

\medskip
\begin{lemma}\label{lem1A} { Let $x_0\in D,$ let $n\geqslant 2,$ let $p>n-1$ and let
$Q:D\rightarrow [0,\infty]$ be locally integrable function with
degree $\alpha:=\frac{n-1}{p-n+1}.$ If $f:D\rightarrow
\overline{{\Bbb R}^n}$ is an open discrete lower $Q$-mapping at
$x_0$ with respect to $p$-modulus, then $f$ satisfies
$${\rm cap}_{\beta}(f(E))\leqslant \frac{\omega_{n-1}}{I^{\,*\,\beta-1}}$$
at the point $x_0$ with $\beta:=\frac{p}{p-n+1}$ and
$Q^{\,*}=Q^{\frac{n-1}{p-n+1}}.$

Here $E$ is the condenser $(B(x_0, r_2), \overline{B(x_0, r_1)}),$
$q^{\,*}_{x_0}(r)$ denotes the spherical mean of
$Q^{\frac{n-1}{p-n+1}}$ over $|x-x_0|=r$ and $I^{\,*}=I^{\,*}(r_1,
r_2)=\int\limits_{r_1}^{r_2}\frac{dt}{t^{\frac{n-1}{\beta-1}}
q_{x_0}^{*\,\frac{1}{\beta-1}}(t)}$ for $0<r_1<r_2<{\rm dist\,}(x_0,
\partial D).$
}\end{lemma}

\medskip In what follows we will need the following auxiliary
assertion (see, for example, \cite[Lemma~7.4, ch.~7]{MRSY$_2$} for
$p=n$ and \cite[Lemma~2.2]{Sal} for $p\ne n.$

\medskip
\begin{proposition}\label{pr1A}
{\, Let $x_0 \in {\Bbb R}^n,$ $Q(x)$ be a Lebesgue measurable
function, $Q:{\Bbb R}^n\rightarrow [0, \infty],$ $Q\in
L_{loc}^1({\Bbb R}^n).$ We set $A:=A(x_0, r_1, r_2)=\{ x\,\in\,{\Bbb
R}^n : r_1<|x-x_0|<r_2\}$ and
$\eta_0(r)=\frac{1}{Ir^{\frac{n-1}{p-1}}q_{x_0}^{\frac{1}{p-1}}(r)},$
where $I:=I=I(x_0,r_1,r_2)=\int\limits_{r_1}^{r_2}\
\frac{dr}{r^{\frac{n-1}{p-1}}q_{x_0}^{\frac{1}{p-1}}(r)}$ and
$q_{x_0}(r):=\frac{1}{\omega_{n-1}r^{n-1}}\int\limits_{|x-x_0|=r}Q(x)\,d{\mathcal
H}^{n-1}$ is the integral average of the function $Q$ over the
sphere $S(x_0, r).$ Then
\begin{equation*}\label{eq10A_1}
\frac{\omega_{n-1}}{I^{p-1}}=\int\limits_{A} Q(x)\cdot
\eta_0^p(|x-x_0|)\ dm(x)\leqslant\int\limits_{A} Q(x)\cdot
\eta^p(|x-x_0|)\ dm(x)
\end{equation*}
for any Lebesgue measurable function $\eta :(r_1,r_2)\rightarrow
[0,\infty]$ such that
$\int\limits_{r_1}^{r_2}\eta(r)\,dr=1. $ Moreover,~(\ref{eq10A_1})
holds for similar functions $\eta$ with
$\int\limits_{r_1}^{r_2}\eta(r)\,dr\geqslant 1$ (see, e.g.,
\cite[Remark~3.1]{Sev$_3$}).}
\end{proposition}

\medskip
Combining Lemma~\ref{lem1A} with Proposition~\ref{pr1A}, we obtain
the following statement.

\medskip
\begin{lemma}\label{lem5} { Let $x_0\in D,$ let $n\geqslant 2,$ let $p>n-1$ and let
$Q:D\rightarrow [0,\infty]$ be locally integrable function with
degree $\alpha:=\frac{n-1}{p-n+1}.$ If $f:D\rightarrow
\overline{{\Bbb R}^n}$ is an open discrete lower $Q$-mapping at
$x_0$ with respect to $p$-modulus, $p>n-1,$ then $f$ satisfies
$${\rm cap}_{\beta}(f(E))\leqslant \int\limits_{A(x_0, r_1, r_2)} Q^{\frac{n-1}{p-n+1}}(x)\cdot
\eta^{\beta}(|x-x_0|)\ dm(x)$$
for any Lebesgue measurable function $\eta :(r_1,r_2)\rightarrow
[0,\infty]$ such that
$\int\limits_{r_1}^{r_2}\eta(r)\,dr\geqslant 1,$ where
$\beta:=\frac{p}{p-n+1}$ and $Q^{\,*}=Q^{\frac{n-1}{p-n+1}}.$ Here
$E$ is the condenser $(B(x_0, r_2), \overline{B(x_0, r_1)}),$
$0<r_1<r_2<{\rm dist\,}(x_0,
\partial D).$
}\end{lemma}

\medskip
The following definition is from \cite[Section~2, Ch.~III]{Ri}. Let
$F$ be a compact set in ${\Bbb R}^n$. We say that $F$ is of {\it
$p$-capacity zero} if ${\rm cap}_p\,(A, F)=0$ for some bounded open
set $A\supset F.$ An arbitrary set $E\subset {\Bbb R}^n$ is of
$p$-capacity zero if the same is true for every compact subset of
$E.$ In this case we write ${\rm cap}_p\, E = 0$ (${\rm cap}\, E=0$
if $p=n$), otherwise ${\rm cap}_p E>0.$ The following statement was
proved in~\cite[Lemma~4]{Sev$_2$} for $p=n$.

\medskip
\begin{lemma}\label{lem6}
{\, Let $n-1<p\leqslant n,$ let $D$ be a domain in ${\Bbb R}^n,$
$n\ge 2,$ let $E\subset \overline{{\Bbb R}^n}$ be a compact set of
positive capacity for $p=n,$ and an arbitrary set for $p\ne n,$ let
$\Phi:\overline{{\Bbb R^{+}}}\rightarrow \overline{{\Bbb R^{+}}}$ be
an increasing convex function, let $M>0$ and let ${\frak F}^{p,
M}_{\Phi, E}(x_0)$ be a family of open discrete mappings
$f:D\,\rightarrow\,\overline{{\Bbb R}^n}\setminus E$ for which there
exists a function $Q=Q_f:D\rightarrow [0, \infty]$ be a Lebesgue
measurable function such that
$${\rm cap}_{p}(f({\mathcal E}))\leqslant \int\limits_{A(x_0, r_1, r_2)} Q(x)\cdot
\eta^p(|x-x_0|)\ dm(x)$$
for any Lebesgue measurable function $\eta :(r_1,r_2)\rightarrow
[0,\infty]$ such that
$\int\limits_{r_1}^{r_2}\eta(r)\,dr\geqslant 1,$ ${\mathcal
E}=(B(x_0, r_2), \overline{B(x_0, r_1)}),$ while
$$\int\limits_{D}\Phi(Q_f(x))\cdot\frac{dm(x)}{(1+|x|^2)^n}\leqslant
M\,.$$
If for some $\delta_0>\tau_0:=\Phi(0),$
$$\int\limits_{\delta_0}^{\infty}
\frac{d\tau}{\tau\left[\Phi^{-1}(\tau)\right]^{\frac{1}{p-1}}}=
\infty\,,$$
then the family of mappings ${\frak F}^{p, M}_{\Phi, E}(x_0)$ is
equicontinuous at the point~$x_0.$ }
\end{lemma}

\medskip
{\it Proof of Lemma~\ref{lem6}} repeats verbatim the proofs of
Theorems~1.1 and 1.2 in \cite{Sev$_5$}, and is therefore omitted.

\medskip
We define for any $x\in D$ and fixed $p,q\geqslant 1$
\begin{equation}\label{eq15.1}
K_{I, q}(x,f)\quad =\quad\left\{
\begin{array}{rr}
\frac{|J(x,f)|}{{l\left(f^{\,\prime}(x)\right)}^q}, & J(x,f)\ne 0,\\
1,  &  f^{\,\prime}(x)=0, \\
\infty, & {\rm otherwise}
\end{array}
\right.\,.\end{equation}

\medskip
Given a mapping $f:D\,\rightarrow\,{\Bbb R}^n,$ a set $E\subset D$
and $y\,\in\,{\Bbb R}^n,$ we define the {\it multiplicity function
$N(y,f,E)$} as a number of preimages of the point $y$ in a set $E,$
i.e.
\begin{equation}\label{eq23}
N(y,f,E)\,=\,{\rm card}\,\left\{x\in E: f(x)=y\right\}\,,
\end{equation}
\begin{equation}\label{eq1G}
N(f,E)\,=\,\sup\limits_{y\in{\Bbb R}^n}\,N(y,f,E)\,.
\end{equation}

\medskip
The following statement was first proved for homeomorphisms and
$x_0\in\overline{D}$ in \cite[Theorem~2.1]{KR},
cf.~\cite[Lemma~2.3]{PSS}.

\medskip
\begin{lemma}{}\label{thOS4.1} {\it\,
Let $D$ be a domain in ${\Bbb R}^n,$ $n\geqslant 3,$
$\varphi:(0,\infty)\rightarrow (0,\infty)$ be a non-decreasing
function satisfying condition
$$
\int\limits_{1}^{\infty}\left(\frac{t}{\varphi(t)}\right)^
{\frac{1}{n-2}}dt<\infty\,.
$$
If $n\geqslant 3$ and $p>n-1,$ then every open discrete mapping
$f:D\rightarrow {\Bbb R}^n$ with finite distortion of the class
$W^{1,\varphi}_{\rm loc}$ such that $N(f, D)<\infty$ is a lower
$Q$-mapping with respect to $p$-modulus at each point
$x_0\in\overline{D}$ for
$$Q(x)=N(f, D)\cdot K^{\frac{p-n+1}{n-1}}_{I, \alpha}(x, f),$$
$\alpha:=\frac{p}{p-n+1},$ where the inner dilatation
$K_{I,\alpha}(x, f)$ of $f$ at $x$ is of order $\alpha$ is defined
by the relation~(\ref{eq15.1}), and the multiplicity $N(f, D)$ is
defined by the second relation in~(\ref{eq1G}).}
\end{lemma}

The following statement holds, see~\cite[Theorem~5]{Sev$_1$}.

\medskip
\begin{lemma}\label{lem4}
{\, Let $x_0\in \partial D,$ let $f:D\rightarrow {\Bbb R}^n$ be a
bounded, open, discrete, and closed lower $Q$-mapping with respect
to $p$-modulus in a domain $D\subset{\Bbb R}^n,$ $Q\in
L_{loc}^{\frac{n-1}{p-n+1}}({\Bbb R}^n),$ $n-1<p,$ and
$\alpha:=\frac{p}{p-n+1}.$ Then, for every
$\varepsilon_0<d_0:=\sup\limits_{x\in D}|x-x_0|$ and every compact
set $C_2\subset D\setminus B(x_0, \varepsilon_0)$ there exists
$\varepsilon_1,$ $0<\varepsilon_1<\varepsilon_0,$ such that, for
each $\varepsilon\in (0, \varepsilon_1)$ and each compact
$C_1\subset \overline{B(x_0, \varepsilon)}\cap D$ the inequality
\begin{equation}\label{eq3A_2}
M_{\alpha}(f(\Gamma(C_1, C_2, D)))\leqslant \int\limits_{A(x_0,
\varepsilon, \varepsilon_1)}Q^{\frac{n-1}{p-n+1}}(x)
\eta^{\alpha}(|x-x_0|)\,dm(x)
\end{equation}
holds, where $A(x_0, \varepsilon, \varepsilon_1)=\{x\in {\Bbb R}^n:
\varepsilon<|x-x_0|<\varepsilon_1\}$ and $\eta: (\varepsilon,
\varepsilon_1)\rightarrow [0,\infty]$ is an arbitrary Lebesgue
measurable function such that
\begin{equation}\label{eq6B}
\int\limits_{\varepsilon}^{\varepsilon_1}\eta(r)\,dr=1\,.
\end{equation}
}
\end{lemma}

\medskip
Given $N\in {\Bbb N},$ $M>0,$ $\delta>0,$ $\alpha\geqslant 1,$
closed sets $E_*, F$ in $\overline{{\Bbb R}^n},$ $n\geqslant 3,$ a
domain $D\subset {\Bbb R}^n,$ a closed (with respect to $D$) set $E$
in $D,$ an increasing function
$\varphi:[0,\infty)\rightarrow[0,\infty)$ and an increasing convex
function $\Phi:\overline{{\Bbb R^{+}}}\rightarrow \overline{{\Bbb
R^{+}}}$ we denote by $\frak{R}^{E_*, E, F, N, M}_{\varphi, \Phi,
\delta, \alpha}(D)$ a (some) family of bounded open discrete
mappings $f:D\rightarrow \overline{{\Bbb R}^n}\setminus F$ for which
there exists a Lebesgue measurable function $Q=Q_f:D\rightarrow [0,
\infty]$ such that $K_{I, \alpha}(x, f)\leqslant Q(x)$ a.e.,
(\ref{eq1D}) holds, $N(f, D)\leqslant N$ and, in addition,

\medskip
1) $C(f, \partial D)\subset E_*,$

\medskip
2) for each component $K$ of $D^{\,\prime}_f\setminus  E_*,$
$D^{\,\prime}_f:=f(D),$ there is a continuum $K_f\subset K$ such
that $h(K_f)\geqslant \delta$ and $h(f^{\,-1}(K_f), \partial
D)\geqslant \delta>0,$

\medskip
3) $f^{\,-1}(E_*)=E.$

\medskip
The following statement holds.

\medskip
\begin{theorem}\label{th3} {\, Let $\alpha\in(n-1, n],$
let $D$ be a bounded domain in ${\Bbb R}^n,$ $n\geqslant 2.$ Assume
that:

\medskip
1) the set $E$ is nowhere dense in $D,$ and $D$ is finitely
connected on $E\cup \partial D,$ i.e., for any $z_0\in E\cup
\partial D$ and any neighborhood $\widetilde{U}$ of $z_0$ there is a
neighborhood $\widetilde{V}\subset \widetilde{U}$ of $z_0$ such that
$(D\cap \widetilde{V})\setminus E$ consists of finite number of
components;

\medskip
2) for any $x_0\in\partial D$ there is $m=m(x_0)\in {\Bbb N},$
$1\leqslant m<\infty$ such that the following is true: for any
neighborhood $U$ of $x_0$ there is a neighborhood $V\subset U$ of
$x_0$ and  such that:

\medskip
2a) $V\cap D$ is connected,

\medskip
2b) $(V\cap D)\setminus E$ consists at most of $m$ components;

\medskip
3) for some $\delta_0>\tau_0:=\Phi(0),$
\begin{equation}\label{eq2A} \int\limits_{\delta_0}^{\infty}
\frac{d\tau}{\tau\left[\Phi^{-1}(\tau)\right]^{\frac{1}{p-1}}}=
\infty;
\end{equation}

\medskip
4) the function $\varphi$ satisfies Calderon condition
\begin{equation}\label{eq1_A_10}
\int\limits_{1}^{\infty}\left(\frac{t}{\varphi(t)}\right)^
{\frac{1}{n-2}}\,dt<\infty\,.
\end{equation}

\medskip
Let for $\alpha=n$ the set $F$ have positive capacity, and for
$n-1<\alpha<n$ it is an arbitrary closed set. Let the family of all
components of $D^{\,\prime}_f\setminus E_*$ is equi-uniform over
$f\in\frak{R}^{E_*, E, F, N, M}_{\varphi, \Phi, \delta, \alpha}(D)$
with respect to $\alpha$-modulus. Then every $f\in\frak{R}^{E_*, E,
F, N, M}_{\varphi, \Phi, \delta, \alpha}(D)$ has a continuous
extension to $\partial D$ and the family $\frak{R}^{E_*, E, F, N,
M}_{\varphi, \Phi, \delta, \alpha}(\overline{D}),$ consisting of all
extended mappings $\overline{f}: \overline{D}\rightarrow
\overline{{\Bbb R}^n},$ is equicontinuous in~$\overline{D}.$ The
equicontinuity must be understood in the sense of the chordal metric
$h.$ }
\end{theorem}

\medskip
\begin{proof}
Let $x_0\in D.$ By Lemma~\ref{thOS4.1}, every $f\in \frak{R}^{E_*,
E, F, N, M}_{\varphi, \Phi, \delta, \alpha}(D)$ is a lower
$Q$-mapping relative to the $p$-modulus at every point
$x_0\in\overline {D}$ for $Q(x)=N\cdot K^{\frac{p-n+1}{n-1}}_{I,
\alpha}(x, f),$ where
$\alpha=\frac{p}{p-n+1}.$ Now, by Lemma~\ref{lem5} $f$ satisfies the
relation
$${\rm cap}_{\alpha}(f({\mathcal{E}}))\leqslant \int\limits_{A(x_0,
r_1, r_2)} N^{\frac{n-1}{p-n+1}}\cdot K_{I, \alpha}(x,
f)\eta^{\,\alpha}(|x-x_0|)\,dm(x)$$
$$\leqslant
 N^{\frac{n-1}{p-n+1}}\cdot\int\limits_{A(x_0, r_1, r_2)} Q(x)\eta^{\,\alpha}(|x-x_0|)\,dm(x)$$
for any Lebesgue measurable function $\eta :(r_1,r_2)\rightarrow
[0,\infty]$ such that
$\int\limits_{r_1}^{r_2}\eta(r)\,dr\geqslant 1.$ Now, the
equicontinuity of the family $\frak{R}^{E_*, E, F, N, M}_{\varphi,
\Phi, \delta, \alpha}(D)$ at inner points $x_0\in D$ follows by
Lemma~\ref{lem6}.

\medskip
Now, we prove the following statement: for any $x_0\in \partial D,$
$x_0\ne \infty,$ and for every $\varepsilon>0$ there is
$\delta=\delta(\varepsilon, x_0)>0$ such that
\begin{equation}\label{eq1B}
h(f(x), f(y))<\varepsilon
\end{equation}
for any $f\in\frak{R}^{E_*, E, F, N, M}_{\varphi, \Phi, \delta,
\alpha}(D)$ and any $x, y\in B(x_0, \delta)\cap D.$

\medskip
Suppose the opposite. Then there is $x_0\in\partial D,$
$\varepsilon_0>0,$ a sequences $x_k, x^{\,\prime}_k\in D,$ $x_k,
x^{\,\prime}_k\rightarrow x_0$ as $k\rightarrow \infty,$ and $f_k\in
\frak{R}^{E_*, E, F, N, M}_{\varphi, \Phi, \delta, \alpha}(D)$ such
that
\begin{equation}\label{eq13D}
h(f_k(x_k), f_k(x^{\,\prime}_k))\geqslant \varepsilon_0/2\,.
\end{equation}
By the assumption~2), there exists a sequence of neighborhoods
$V_k\subset B(x_0, 2^{\,-k}),$ $k=1,2,\ldots ,$ such that $V_k\cap
D$ is connected and $(V_k\cap D)\setminus E$ consists of $m$
components, $1\leqslant m<\infty.$

Arguing similarly to the proof of Lemma~\ref{lem1}, we may consider
that $x_k, x^{\,\prime}_k\not\in E.$ Now, by Lemma~\ref{lem2} there
are subsequences $x_{k_l}$ and $x^{\,\prime}_{k_l},$ $l=1,2,\ldots
,$ belonging to some sequence of neighborhoods $V_l,$ $l=1,2,\ldots
,$ of the point $x_0$ such that ${\rm diam\,}V_l\rightarrow 0$ as
$l\rightarrow\infty$ and, in addition, any pair $x_{k_l}$ and
$x^{\,\prime}_{k_l}$ may be joined by a path $\gamma_l$ in $V_l\cap
D,$ where $\gamma_l$ contains at most $m-1$ points in $E.$ Without
loss of generality, we may assume that the same sequences $x_k$ and
$y_k$ satisfy properties mentioned above. Let $\gamma_k:[0,
1]\rightarrow D,$ $\gamma_k(0)=x_k$ and
$\gamma_k(1)=x^{\,\prime}_k,$ $k=1,2,\ldots .$

\medskip
Observe that, the path $f_k(\gamma_k)$ contains not more than $m-1$
points in $E_*.$ Let
$$b_{1}=f_k(\gamma_k(t_1)), b_{2}=f_k(\gamma_k(t_2))\quad,\ldots,\quad
b_l=f_k(\gamma_k(t_l))\,,$$ $$t_0:=0\leqslant t_1\leqslant
t_2\leqslant \ldots\leqslant t_l\leqslant 1=t_{l+1},\qquad
0\leqslant l\leqslant m-1\,,$$
be points in $f_k(\gamma_k)\cap E_*.$ By the relation~(\ref{eq13A})
and due to the triangle inequality,
\begin{equation}\label{eq6C_1}
\varepsilon_0/2\leqslant h(f_k(x_k),
f_k(x^{\,\prime}_k))\leqslant\sum\limits_{r=0}^{l}
h(f_k(\gamma_k(t_r)), f_k(\gamma_k(t_{r+1}))\,.
\end{equation}
It follows from~(\ref{eq6C_1}) that, there is $1\leqslant
r=r(k)\leqslant m-1$ such that
\begin{equation}\label{eq7F_1}
h(f(\gamma_k(t_{r(k)})), f_k(\gamma_k(t_{r(k)+1})))\geqslant
\varepsilon_0/(2(l+1))\geqslant \varepsilon_0/(2m)\,.
\end{equation}
Observe that, due to~(\ref{eq7F_1}) the set
$G_k:=|f(\gamma_k)|_{(t_{r(k)}, t_{r(k)+1})}|$ belongs to
$D^{\,\prime}\setminus E_*$ and that $G_k$ contains some a continuum
$\widetilde{G}_k$ with $h(\widetilde{G}_k)\geqslant
\varepsilon_0/(4m)$ for any $k\in {\Bbb N}.$

Since $\widetilde{G}_k$ is a continuum in
$D^{\,\prime}_{f_k}\setminus E_*,$ there is a component $K_k$ of
$D^{\,\prime}_{f_k}\setminus E_*,$ containing $K_k.$ Let us apply
the definition of equi-uniformity for he sets $\widetilde{G}_k$ and
$K_{f_k}$ in $K_k$ (here $K_{f_k}$ is a continuum from the
definition of the class $\frak{R}^{E_*, E, F, N, M}_{\varphi, \Phi,
\delta, \alpha}(D)$, in particular, $h(K_{f_k})\geqslant \delta$).
Due to this definition, for the number $\delta_*:=\min\{\delta,
\varepsilon/4m\}>0$ there is $P>0$ such that
\begin{equation}\label{eq1F_1}
M_p(\Gamma(\widetilde{G}_k, K_{f_k}, K_k))\geqslant P
>0\,, \qquad k=1,2,\ldots\,.
\end{equation}
Let us to show that, the relation~(\ref{eq1F_1}) contradicts with
the definition of $\frak{R}^{E_*, E, F, N, M}_{\varphi, \Phi,
\delta, \alpha}(D)$ together with conditions~(\ref{eq1D}),
(\ref{eq2A}). Indeed, let us denote by $\Gamma_k$ the family of all
half-open paths $\beta_k:[a, b)\rightarrow \overline{{\Bbb R}^n}$
such that $\beta_k(a)\in \widetilde{G}_k,$ $\beta_k(t)\in K_k$ for
all $t\in [a, b)$ and, moreover, $\lim\limits_{t\rightarrow
b-0}\beta_k(t):=B_k\in K_{f_k}.$ Obviously, by~(\ref{eq1F_1})
\begin{equation}\label{eq4C_1}
M_p(\Gamma_k)=M_p(\Gamma(\widetilde{G}_k, K_{f_k}, K_k))\geqslant P
>0\,, \qquad k=1,2,\ldots\,.
\end{equation}
As under the proof of Lemma~\ref{lem1}, we may prove that
$D\setminus E$ consists of finite number of components $D_1,\ldots,
D_s,$ $1\leqslant s<\infty,$ while $f_k$ is closed in each $D_i,$
$i=1,2,\ldots, $ i.e. $f_k(E)$ is closed in $f_k(D_i)$ whenever $E$
is closed in $D_i.$ We also may show that $f_k(D_i)=K$ for some
$i\in {\Bbb N}.$ Without loss of generality, we may consider that
$\nabla_k:=\gamma_k|_{[p_{s(k)}, q_{s(k)}]}$ belong to (some) one
component $D_1$ of $D\setminus E$ and $f_k(D_1)=K,$ where
$t_{r(k)}<p_{s(k)}<q_{s(k)}<t_{r(k)+1},$ $a_{s,
k}:=f_k(\gamma_k(p_s))\rightarrow w_k:=f_k(\gamma_k(t_{r(k)}))$ as
$s\rightarrow \infty$ and $b_{s, k}=f_k(\gamma_k(q_s))\rightarrow
f_k(\gamma_k(t_{r(k)+1}))$ as $s\rightarrow \infty.$ Consider the
family $\Gamma_k^{\,\prime}$ of all maximal $f_k$-liftings
$\alpha_k:[a, c)\rightarrow D$ of the family $\Gamma_k$ starting at
$|\nabla_k|;$ such a family exists by Proposition~\ref{pr3}.

Observe that, the situation when $\alpha_k\rightarrow \partial D_1$
as $k\rightarrow\infty$ is impossible. Suppose the opposite: let
$\alpha_k(t)\rightarrow \partial D_1$ as $t\rightarrow c.$ We choose
an arbitrary sequence $\varphi_m\in [0, c)$ such that
$\varphi_m\rightarrow c-0$ as $m\rightarrow\infty.$ Since the space
$\overline{{\Bbb R}^n}$ is compact, the boundary $\partial D_1$ is
also compact as a closed subset of the compact space. Then there
exists $w_m\in
\partial D_1$ such that
\begin{equation}\label{eq7G_1}
h(\alpha_k(\varphi_m), \partial D_1)=h(\alpha_k(\varphi_m), w_m)
\rightarrow 0\,,\qquad m\rightarrow \infty\,.
\end{equation}
Due to the compactness of $\partial D_1$, we may assume that
$w_m\rightarrow w_0\in \partial D_1$ as $m\rightarrow\infty.$
Therefore, by the relation~(\ref{eq7G_1}) and by the triangle
inequality
\begin{gather}
h(\alpha_k(\varphi_m), w_0)\nonumber \\ \label{eq8C_1} \leqslant
h(\alpha_k(\varphi_m), w_m)+h(w_m, w_0)\rightarrow 0\,,\qquad
m\rightarrow \infty\,.
\end{gather}
There are two cases: $w_0\in \partial D,$ or $w_0\in D.$ In the
first case,
\begin{equation}\label{eq9C_1}
f_k(\alpha_k(\varphi_m))=\beta_k(\varphi_m)\rightarrow \beta(c)
\,,\quad m\rightarrow\infty\,,
\end{equation}
because by the construction the path $\beta_k(t),$ $t\in [a, b],$
lies in $K_k\subset D^{\,\prime}_{f_k}\setminus E_*$ together with
its finite ones points. At the same time, by~(\ref{eq8C_1})
and~(\ref{eq9C_1}) we have that $\beta_k(c)\in C(f_k,
\partial D)\subset E_*$ by the definition of the class
$\frak{R}^{E_*, E, F, N, M}_{\varphi, \Phi, \delta, \alpha}(D)$ and
because $C(f_k, \partial D)\subset E_*.$ The inclusions
$\beta_k\subset D^{\,\prime}\setminus E_*$ and $\beta_k(c)\in E_*$
contradict each other.

\medskip
In the second case, when $w_0\in D,$ we have that $w_0\in E.$
However,
$f_k(\alpha_k(\varphi_m))=\beta_k(\varphi_m)\rightarrow\beta_k(c)$
and consequently, $\beta_k(c)\in f_k(E).$ But now $\beta_k(c)\in
E_*$ because $f_k^{\,-1}(E_*)=E.$ The latter contradicts the
definition of $\beta_k.$ Thus, the case $\alpha_k\rightarrow
\partial D_1$ as $k\rightarrow\infty$ is impossible, as required.

\medskip
Therefore, by Proposition~\ref{pr3} $\alpha_k\rightarrow x_1\in D_1$
as $t\rightarrow c-0,$ and $c_b$ and $f_k(\alpha_k(b))=f_k(x_1).$ In
other words, the $f_k$-lifting $\alpha_k$ is complete, i.e.,
$\alpha_k:[a, b]\rightarrow D.$ Besides that, it follows from that
$\alpha_k(b)\in f_k^{\,-1}(K_{f_k}).$

\medskip
Again, by the definition of the class $\frak{R}^{E_*, E, F, N,
M}_{\varphi, \Phi, \delta, \alpha}(D),$
\begin{equation}\label{eq13B_1}
h(f_k^{\,-1}(K_{f_k}), \partial D)\geqslant \delta>0\,.
\end{equation}
Since $x_0\ne\infty,$ it follows from~(\ref{eq13B_1}) that
\begin{equation}\label{eq14C_1}
f_k^{\,-1}(K_{f_k})\subset D\setminus B(x_0, r_0)
\end{equation}
for any $k\in {\Bbb N}$ and some $r_0>0.$ Let $k$ be such that
$2^{\,-k}<\varepsilon_0.$ We may consider that $r_0<\varepsilon_0,$
where $\varepsilon_0$ is a number from the conditions of the lemma.
Applying now Lemmas~\ref{thOS4.1} and~\ref{lem4} together
with~(\ref{eq14C_1}), we obtain for $\alpha:=\frac{p}{p-n+1}$ that
\begin{gather}
M_{\alpha}(f(\Gamma_k^{\,\prime}))=M_{\alpha}(\Gamma(\widetilde{G}_k,
C_0^{\,\prime}, K))\nonumber \\ \label{eq15A_1} \leqslant
N^{\frac{n-1}{p-n+1}}\int\limits_{A(x_0, 2^{\,-k},
r_0)}Q_{f_k}(x)\eta^{\,\alpha}(|x-x_0|)\,dm(x)
\end{gather}
for sufficiently large $k\in {\Bbb N}$ and for any nonnegative
Lebesgue measurable function $\eta$ satisfying the
relation~(\ref{eqA2}) with $\varepsilon=2^{\,-k}$ and
$\varepsilon_0=r_0.$ Now, we set
\begin{equation}\label{eq9_1} I_k=I(x_0, 2^{\,-k},
r_0)=\int\limits_{2^{\,-k}}^{r_0}\
\frac{dr}{r^{\frac{n-1}{p-1}}\widetilde{q}_{k,
x_0}^{\frac{1}{p-1}}(r)}\,,
\end{equation}
where
$$\widetilde{q}_{k, x_0}(r)=\frac{1}{\omega_{n-1}r^{n-1}} \int\limits_{S(x_0,
r)}\widetilde{Q}_{f_{m_k}}(x)\,d\mathcal{H}^{n-1}\,$$ and
$$\widetilde{Q}_{f_{m_k}}(x)=\begin{cases}
Q_{f_{m_k}}(x), & Q(x)\geqslant 1\\ 1, & Q(x)<1\end{cases}\,.$$

By Lemma~\ref{lem1} and Remark~\ref{rem1}
\begin{equation}\label{eq2_1_1}
I_k=\int\limits_{2^{\,-k}}^{r_0}\
\frac{dr}{r^{\frac{n-1}{p-1}}\widetilde{q}_{k,
x_0}^{\frac{1}{p-1}}(r)}\rightarrow\infty
\end{equation}
as $k\rightarrow\infty.$
Since $\widetilde{q}_{k, x_0}(r)\geqslant 1$ for a.e. $r$ and
by~(\ref{eq2_1_1}), $0<I_k<\infty$ for sufficiently large $k\in
{\Bbb N}.$ Set
$$\psi_k(t)= \left \{\begin{array}{rr}
1/[t^{\frac{n-1}{p-1}}\widetilde{q}_{k, x_0}^{\frac{1}{p-1}}(t)], &
t\in (2^{\,-k}, r_0)\ ,
\\ 0,  &  t\notin (2^{\,-k}, r_0)\ .
\end{array} \right. $$
Let $\eta_k(t)=\psi_k(t)/I_k$ for $t\in (2^{\,-k}, r_0)$ and
$\eta_k(t)=0$ otherwise. Now, $\eta_k$ satisfies~(\ref{eqA2}) for
$r_1=2^{\,-k}$ and $r_2=r_0.$ Now, by~(\ref{eq2*!A}),
(\ref{eq15A_1}), (\ref{eq2_1_1}) and by Fubini's theorem,
\begin{equation}\label{eq3_1_1}
M_p(f_k(\Gamma^{\,\prime}_k)\leqslant\frac{N^{\frac{n-1}{p-n+1}}}{I^{p}_k}\int\limits_{A(x_0,
2^{\,-k}, r_0)}
\widetilde{Q}_{f_{m_k}}(x)\cdot\psi_k^p(|x-x_0|)\,dm(x)=\frac{\omega_{n-1}}{I^{p-1}_k}\rightarrow
0
\end{equation}
as $k\rightarrow\infty.$ Thus, by~(\ref{eq3_1_1}) and
(\ref{eq15A_1})--(\ref{eq9_1}) we obtain that
\begin{equation}\label{eq3A_3}
M_p(\Gamma_k)= M_p(f_k(\Gamma_k^{\,\prime}))\leqslant
\Delta(k)\rightarrow 0
\end{equation}
as $k\rightarrow \infty.$ However, the relation~(\ref{eq3A_3})
together with the inequality~(\ref{eq4C_1}) contradict each other,
which proves~(\ref{eq1B}).

\medskip
Observe that, the possibility of a continuous extension of any $f\in
\frak{R}^{E_*, E, F, N, M}_{\varphi, \Phi, \delta, \alpha}(D)$ at
$x_0,$ $x_0\ne \infty,$ follows from~(\ref{eq1B}), as well. Indeed,
otherwise we have $h(f(x_k), f(y_k))\geqslant \varepsilon_0$ for
some sequences $x_k, y_k\rightarrow x_0$ as $k\rightarrow\infty,$
but this contradicts with~(\ref{eq1B}).

\medskip
Now, we prove that $\frak{R}^{E_*, E, F, N, M}_{\varphi, \Phi,
\delta, \alpha}(D)$ is equicontinuous at $x_0$ with a respect to
$\overline{D}.$ Suppose the opposite. Then there is $x_0\in\partial
D,$ $\varepsilon_0>0,$ a sequence $x_k\rightarrow x_0,$
$x_k\in\overline{D},$ and $f_k\in \frak{R}^{E_*, E, F, N,
M}_{\varphi, \Phi, \delta, \alpha}(\overline{D})$ such that
\begin{equation}\label{eq12D}
h(f_k(x_k), f_k(x_0))\geqslant \varepsilon_0\,.
\end{equation}
Since $f_k$ has a continuous extension to $\partial D,$ we may
consider that $x_k\in D.$ In addition, it follows from~(\ref{eq12D})
that we may find a sequence $x^{\,\prime}_k\in D,$ $k=1,2,\ldots ,$
such that $x^{\,\prime}_k\rightarrow x_0$ and
\begin{equation}\label{eq13E}
h(f_k(x_k), f_k(x^{\,\prime}_k))\geqslant \varepsilon_0/2\,.
\end{equation}
The relation~(\ref{eq13E}) contradicts with~(\ref{eq1B}). The
contradiction obtained above completes the proof for $x_0\ne\infty.$

\medskip
Finally, if $x_0=\infty,$ we set $\widetilde{f}_k:=\psi\circ f_k,$
$k=1,2,\ldots ,$ where $\psi(y)=\frac{y}{|y|^2}.$ By the assumption
of Theorem~\ref{th3}, the mappings $\widetilde{f}_k$ satisfy the
relations~(\ref{eq2*!A})--(\ref{eqA2}) at the origin with a new
function $\widetilde{Q}(x):=Q_{f_k}\left(\frac{x}{|x|}\right).$
The conformal change of the variables $x=\frac{y}{|y|^2}$
corresponds to the jacobian $J(x, y)=\frac{1}{|y|^{2n}}.$ So,
observe that, the relation~(\ref{eq3A_1}) holds. So, we may apply
above technique to $Q_{f_k}\left(\frac{x}{|x|^2}\right)$ instead of
$Q_{f_k}(x).$ Repeating further all the arguments given from
relation~(\ref{eq1B}) and further for the family of mappings
$\widetilde{f}_k$ instead of $f_k$, we arrive at the contradiction
with~(\ref{eq1B}), which completes the proof of Theorem~\ref{th3}.
\end{proof}
~$\Box$

\medskip\medskip\medskip
{\bf Funding.} The work was supported by the National Research
Foundation of Ukraine (Project ``Analogues of Carath\'{e}odory and
Koebe-Bloch theorems for Orlycz-Sobolev classes'', Project number
2025.02/0010).

\medskip
{\bf \noindent Victoria Desyatka} \\
Zhytomyr Ivan Franko State University,  \\
40 Velyka Berdychivs'ka Str., 10 008  Zhytomyr, UKRAINE \\
victoriazehrer@gmail.com

\medskip
{\bf \noindent Oleksandr Dovhopiatyi} \\
Zhytomyr Ivan Franko State University,  \\
40 Velyka Berdychivs'ka Str., 10 008  Zhytomyr, UKRAINE \\
Alexdov1111111@gmail.com

\medskip
{\bf \noindent Evgeny Sevost'yanov} \\
{\bf 1.} Zhytomyr Ivan Franko State University,  \\
40 Velyka Berdychivs'ka Str., 10 008  Zhytomyr, UKRAINE \\
{\bf 2.} Institute of Applied Mathematics and Mechanics\\
of NAS of Ukraine, \\
19 Henerala Batyuka Str., 84 116 Slov'yans'k,  UKRAINE\\
esevostyanov2009@gmail.com

\end{document}